\documentclass[12pt]{amsart}   
\usepackage[english]{babel}

\usepackage{amssymb,amscd, graphics,color,latexsym,cancel}   
\usepackage[linktoc=all, pagebackref, hyperindex]{hyperref}%
\hypersetup{colorlinks,  citecolor=blue,  filecolor=blue,  linkcolor=blue,  urlcolor=black}

\usepackage{mathrsfs}
\usepackage[all]{xy}

\usepackage{amsfonts}
\usepackage[normalem]{ulem}

\usepackage{euscript}
\usepackage{amsmath}
\usepackage{ifpdf}

\LARGE\textwidth=6in
\newcommand{\injects}{\hookrightarrow}

\newtheorem{Theorem}{Theorem}[section]
\newtheorem{Lemma}[Theorem]{Lemma}
\newtheorem{Corollary}[Theorem]{Corollary}
\newtheorem{Proposition}[Theorem]{Proposition}

\theoremstyle{definition}
\newtheorem{Remark}[Theorem]{Remark}
\newtheorem{Example}[Theorem]{Example}

\def\sqr#1#2{{\vcenter{\hrule height.#2pt
			\hbox{\vrule width.#2pt height#1pt \kern#1pt
				\vrule width.#2pt}
			\hrule height.#2pt}}}
\def\phi{\varphi}

\def\gr#1#2{{\rm gr}\, _{#1}(#2)}
\def\gr{{\rm gr}\,}

\def\depth{{\rm depth}\,}

\def\Min{{\rm Min}\,}
\def\codim{{\rm codim}\,}

\def\grade{{\rm grade}\,}
\def\rk{\rm rank}

\def\Ext#1#2#3#4{{\rm Ext}\,^{#1}_{#2}({#3},{#4})}

\def\proj#1{{\rm Proj}\, (#1)}
\def\supp#1{{\rm Supp}\, (#1)}

\def\ini{\mbox{\rm in}}

\def\cl#1{{\mathcal #1}}

\def\phi{\varphi}

\def\grade{{\rm grade}\,}

\def\gg{{\bf g}}

\def\gg{{\bf g}}

\def\cl#1{{\cal #1}}
\def\rk{\rm rank}
\def\GG{{\mathbb G}}
\newcommand{\excise}[1]{}

 \def\NZQ{\mathbb}               

 \def\PP{{\NZQ P}}

 \def\GG{{\NZQ G}}

 \def\G{{\mathcal G}}

 \def\opn#1#2{\def#1{\operatorname{#2}}} 
 \opn\chara{char} \opn\length{\ell} \opn\pd{pd} \opn\rk{rk}
 \opn\projdim{proj\,dim} \opn\injdim{inj\,dim} \opn\rank{rank}
 \opn\depth{depth} \opn\grade{grade} \opn\height{height}
 \opn\embdim{emb\,dim} \opn\codim{codim}
 
 \opn\Tr{Tr} \opn\bigrank{big\,rank}
 \opn\superheight{superheight}\opn\lcm{lcm}
 \opn\trdeg{tr\,deg}
 \opn\reg{reg} \opn\lreg{lreg} \opn\ini{in} \opn\lpd{lpd}
 \opn\size{size} \opn\sdepth{sdepth}
 \opn\link{link}\opn\fdepth{fdepth}\opn\lex{lex}
 \opn\tr{tr}
 \opn\type{type}
 \opn\div{div} \opn\Div{Div} \opn\cl{cl} \opn\Cl{Cl}
 \opn\Spec{Spec} \opn\Supp{Supp} \opn\supp{supp} \opn\Sing{Sing}
 \opn\Ass{Ass} \opn\Min{Min}\opn\Mon{Mon}
 \opn\Ann{Ann} \opn\Rad{Rad} \opn\Soc{Soc}
 \opn\Im{Im} \opn\Ker{Ker} \opn\Coker{Coker} \opn\Am{Am}
 \opn\Hom{Hom} \opn\Tor{Tor} \opn\Ext{Ext} \opn\End{End}
 \opn\Aut{Aut} \opn\id{id}
 
 \opn\nat{nat}
 \opn\pff{pf}
 \opn\Pf{Pf} \opn\GL{GL} \opn\SL{SL} \opn\mod{mod} \opn\ord{ord}
 \opn\Gin{Gin} \opn\Hilb{Hilb}\opn\sort{sort}
 \opn\PF{PF}\opn\Ap{Ap}
 \opn\aff{aff} \opn
 \con{conv} \opn\relint{relint} \opn\st{st}
 \opn\lk{lk} \opn\cn{cn} \opn\core{core} \opn\vol{vol}  \opn\inp{inp} \opn\nilpot{nilpot}
 \opn\link{link} \opn\star{star}\opn\lex{lex}\opn\set{set}
 \opn\width{wd}
 \opn\Fr{F}
 \opn\QF{QF}
 \opn\G{G}
 \opn\type{type}\opn\res{res}
 \opn\log{Log}
 \opn\gr{gr}
 
 \def\pot#1#2{#1[\kern-0.28ex[#2]\kern-0.28ex]}

 \opn\dirlim{\underrightarrow{\lim}}
 \opn\inivlim{\underleftarrow{\lim}}
 \let\sect=\cap

\begin{document}

\title{On the Gauss algebra of  toric algebras }

\author{J\"{u}rgen Herzog}
\address{J\"urgen Herzog, Fachbereich Mathematik, Universit\"at Duisburg-Essen, Campus Essen, 45117
Essen, Germany} \email{juergen.herzog@uni-essen.de}

\author{Raheleh Jafari}
\address{Raheleh Jafari, Mosaheb Institute of Mathematics, Kharazmi University,  and School of Mathematics, Institute for Research in Fundamental Sciences (IPM), P.O. Box 19395-5746,	Tehran, Iran.}
\email{rjafari@khu.ac.ir}

\author{Abbas Nasrollah Nejad}
\address{Abbas Nasrollah Nejad, Department of Mathematics, Institute for Advanced Studies in Basic Sciences (IASBS), Zanjan 45137-66731,
	Iran}
\email{abbasnn@iasbs.ac.ir}

\thanks{Essential parts of the paper were written while the authors visited the Mathematische Forschungsinstitut in Oberwolfach in the frame of the ``Research in Pairs" program. We thank the institute  for its generous  support.\\ The second author was in part supported by a grant from IPM (No. 96130112).}



\maketitle

\begin{abstract}
Let $A$ be a $K$-subalgebra of the  polynomial ring $S=K[x_1,\ldots,x_d]$ of dimension $d$, generated by finitely many monomials  of degree $r$.   Then the Gauss algebra $\mathbb{G}(A)$  of $A$ is generated by monomials  of degree $(r-1)d$ in $S$. We describe the generators and the structure of $\mathbb{G}(A)$, when $A$ is a Borel fixed algebra, a squarefree Veronese algebra,  generated in degree $2$, or the edge ring of a bipartite graph with at least   one loop.  For a bipartite graph  $G$  with one loop, the embedding dimension of $\mathbb{G}(A)$ is bounded by the complexity of the graph $G$.
\end{abstract}

\section*{Introduction}

Let $V\subseteq \PP_K^{n-1}$ be a projective variety of dimension  {$d-1$} over an algebraically closed field $K$ of characteristic zero. Denote by  $V_{\mathrm{sm}}$ the set of non-singular
points of $V$ and by $\GG( {d-1},n-1)$ the Grassmannian of  {$d-1$}-planes in $ {\PP_K^{n-1}}$.
The \textit{ Gauss map}  of $V$ is the morphism   $$\gamma: V_{\mathrm{sm}}\longrightarrow \GG( {d-1},n-1),$$
which sends each point $p\in V_{\mathrm{sm}}$ to
the embedded tangent space $\mathrm{T}_{p}V  $ of $V$ at the point $p$. The closure of the image of $\gamma$ in $\GG( {d-1},n-1)$ is called the \textit{Gauss image} of
$V$, or \textit{the variety of tangent planes} and  is denoted by $\gamma(V)$. The homogeneous coordinate ring of $\gamma(V)$  in the Pl\"{u}cker embedding of the Grassmannian $\mathbb{G}( {d-1},n-1)$ of $ {(d-1)}$-planes is called the \textit{Gauss algebra} of $V$. The Gauss map is a classical subject in algebraic geometry  and has been studied by many authors. For example, it is  known that the  Gauss map of a smooth projective variety is finite \cite{GH,Zak}; in particular, a smooth variety and its Gauss image have the same dimension with the obvious exception of a linear space. Zak \cite[Corollary 2.8]{Zak}  showed that, provided $V$ is not a linear subvariety of $\PP_K^n$, the dimension of the Gauss image satisfies the inequality   $\dim V-\dim  \Sing(V)-1\leq \dim \gamma(V)\leq \dim V$, where $\Sing(V)$ denotes the singular locus of $V$.  For an algebraic proof of Zak's inequality,  see \cite{AKB}.

We take up the situation where $V\subset \PP_K^{n-1}$ is a unirational variety. To elaborate  on the algebraic side of the picture, consider the polynomial ring
$S=K[x_1,\ldots,x_d]$.  Let  $\gg=g_1,\ldots,g_n$ be a sequence of non-constant homogeneous polynomials
of the same degree
 in $S$ generating  the $K$-subalgebra $A=K[\gg]\subseteq S$  of dimension  $d$.  Then  the Jacobian matrix $\Theta(\gg)$ of $\gg$ has rank $d$ \cite[Proposition 1.1]{Aron1}. In this situation we define the  \textit{Gauss algebra} associated with $\gg$ as the
$K$-subalgebra generated by the set of $d\times d$ minors of $\Theta(\gg)$ \cite[Definition
2.1]{BGS}.  Since the definition does not depend on the choice of the  homogeneous generators of $A$, we simply denote the Gauss algebra associated with $\gg$, by $\GG(A)$, and call it the Gauss algebra of $A$.
The Gauss algebra  $\mathbb{G}(A)$ is isomorphic to the coordinate ring of  the Gauss image of the projective variety  defined
parametrically by $\gg$ in the Pl\"{u}cker embedding of the Grassmannian $\mathbb{G}( {d-1},n-1)$ of  {$d-1$}-planes.  Moreover, there is an injective homomorphism of
$K$-algebras $\GG(A)\injects A$ inducing the rational map from $\proj{A}$ to its Gauss image \cite[Lemma 2.3]{BGS}.

In this paper, we study the  Gauss algebra of  toric algebras.
If $A\subset S$ is a
toric algebra with monomial generators $\gg=g_1,\ldots,g_n$
of the same degree,   then all  minors of $\Theta(\gg)$ are monomials. In
particular, the Gauss algebra is a toric algebra. For example, it has been shown that the   Gauss algebra of a Veronese algebra is again  Veronese \cite[Proposition 3.2]{BGS}.
Veronese algebras are special cases  of a more general  class of algebras,   namely the class of Borel fixed algebras. As a generalization of the above mentioned result, we show that the Gauss algebra of any Borel fixed algebra is again Borel fixed, see Theorem~\ref{Borel_fix}. This approach provides a simple proof for
\cite[Proposition 3.2]{BGS}.
Veronese algebras are actually principal Borel fixed algebras, that is, the Borel set defining  the algebra  admits precisely one Borel generator. In general the number of Borel generators of the Borel fixed algebra $A$ and that of $\GG(A)$ may be different. However, in Theorem~\ref{principal} we show that the Gauss algebra of a principal Borel fixed algebra is again principal. This has the nice consequence that the  Gauss algebra of a principal Borel fixed algebra is a normal Cohen-Macaulay domain, and its defining ideal is generated by quadrics. Note that in general the property of $A$ being normal  does not imply that $\GG(A)$ is normal, and vice versa (Example~\ref{non-nomral} and Theorem~\ref{2-Veronese}(d)).

 The Gauss algebra of a squarefree Veronese algebra is much harder to understand. We can give a full description of  $\GG(A)$, when $A$ is a  squarefree Veronese algebra generated in degree $2$.  In  Theorem~\ref{2-Veronese} we show that $\GG(A)$ is defined by all monomials $u$ of degree $d$ and $|\supp(u)|\geq 3$, provided $d\geq 5$. Algebras of this type may be viewed as the base ring of a polymatroid.  In particular,  $\GG(A)$ is normal and Cohen--Macaulay. However, $\GG(A)$ is not normal for $d=4$. Yet for any $d$,  the Gauss map $\gamma: \proj{A}\dashrightarrow \proj{\GG(A)}$ is birational.

In the last section of this paper we study the Gauss algebra of the edge ring of a finite graph. Let $G$ be a loop-less connected graph with $d$ vertices.  It is well-known that the dimension of  the edge ring  $A=K[G]$ of $G$  is  $d$, if $G$ is not bipartite, and is $d-1$ if $G$  is bipartite. In our setting, $\GG(A)$ is defined under the assumption that $\dim A=d$. By using a well-known theorem \cite{GKS} of graph theory,  the generators of $\GG(A)$,  when $G$ is not bipartite, correspond to $d$-sets $E$ of edges of $G$, satisfying the property that the subgraph with edges $E$ has an odd cycle in each of its connected components. In the bipartite case we form the graph $G^L$,  where $L$ is a non-empty subset of the vertex set of $G$, by adding a loop to $G$  for each vertex   in $L$. Then $A=K[G^L]$ has dimension $d$, and there is bijective  map from  the set of pairs $(V,T)$ to the set of monomial  generators of  $\GG(A)$, where $V$ is a non-empty subset of $L$  and $T$ is a set of edges which form  a spanning forest  $G(T)$ of $G$  with the property that   each connected component of $G(T)$ contains exactly one vertex of $V$. From this description it follows that if $|L|=1$,  then the embedding dimension of the Gauss algebra  is bounded by  the complexity of the graph, which by definition, is  the number of spanning trees of the graph. This is an important graph invariant. The number of spanning trees provides a measure for the global reliability of a network. For a complete bipartite graph $K_{m;n}$ the embedding dimension of the Gauss algebra is $\binom{n+m-2}{n-1}\binom{n+m-2}{m-1}$, see Example~\ref{bipartite}, while  the number of spanning tress is $m^{n-1}n^{m-1}$, see for instance \cite[Theorem 1]{HW}. In general the defining ideal of the Gauss algebra admits many binomial generators. Thus it is not surprising that the Gauss algebra is rarely a hypersurface ring. This is for example the case,  when $G$ is a cycle with one loop or a path graph with two loops attached. The Gauss algebra of an odd (resp.\  even) cycle of length $d$ with one loop attached is a hypersurface ring of dimension $d$ (resp.\ $d-1$). More generally, we expect that if  $G$ is a bipartite graph on $[d]$, $L=\{i\}$ and $A$ be the edge ring of $G^L$,  then $\GG(A)$ is a hypersurface ring of dimension $d-1$, if and only if $G$ is an  even cycle.

\section{Toric algebras}

In this section, we collect some basic fact {s} about the Gauss algebra of a toric algebra.
Let $S=K[x_1,\ldots,x_d]$ be a polynomial ring  over $K$, where $K$ is a field of characteristic  zero. Let $\gg=g_1,\ldots,g_n$ be a sequence of monomials with $g_i=  {x_1^{a_{1i}}\cdots x_d^{a_{di}}}$ for $i=1,\ldots,n$.
We associate to the sequence
$\gg$ two matrices, namely  $\Theta(\gg)$ and $\mathrm{Log}(\gg)$,  where $\Theta(\gg)$ is the Jacobian matrix of $\gg$ and  $\mathrm{Log}(\gg)=(a_{ij})$ is the exponent
matrix (or log-matrix ) of $\gg$, whose columns  are the exponent vectors of the  monomials in $\gg$. We denote the  $r$-minor
\[
\det
\begin{bmatrix}
\frac{\partial g_{i_1}}{\partial x_{j_1}} &  \cdots &  \frac{\partial g_{i_1}}{\partial x_{j_r}}  \\
\vdots & \ddots & \vdots \\
\frac{\partial g_{i_r}}{\partial x_{j_1}} & \cdots &  \frac{\partial g_{i_r}}{\partial x_{j_r}}\\
\end{bmatrix}=
\det \begin{bmatrix}
a_{i_1j_1}\frac{g_{i_1}}{x_{j_1}} &\cdots & a_{i_1j_r}\frac{g_{i_1}}{x_{j_r}} \\
\vdots & \ddots & \vdots \\
a_{i_rj_1}\frac{g_{i_r}}{x_{j_1}} &\cdots & a_{i_rj_r}\frac{g_{i_r}}{x_{j_r}} \\
\end{bmatrix}
\]
by $[i_1,\ldots,i_r\ | \ j_1,\ldots,j_r ]_{\Theta(\gg)}$.

\medskip
The multi-linearity property of the determinant implies that
\[x_{j_1}\ldots x_{j_r} [i_1,\ldots,i_r\ | \ j_1,\ldots,j_r ]_{\Theta(\gg)}=g_{i_1}\ldots g_{i_r}[i_1,\ldots,i_r\ | \ j_1,\ldots,j_r ],\]
where  {$[i_1,\ldots,i_r\ | \ j_1,\ldots,j_r ]$ is the $r$-minor corresponding to the rows $i_1,\ldots,i_r$ and columns $j_1,\ldots,j_r$ of the transpose} of $\mathrm{Log}(\gg)$. Therefore,  $r$-minors
of $\Theta(\gg)$ are monomials of the form
\begin{equation}\label{minor-monomail}
[i_1,\ldots,i_r\ | \ j_1,\ldots,j_r ]\cdotp\frac{g_{i_1}\ldots g_{i_r}}{x_{j_1}\ldots x_{j_r}}.
\end{equation}
 By the relation (\ref{minor-monomail}), the Jacobian matrix  and the log-matrix of $\gg$ have the same rank (see also \cite[Proposition~1.2]{Simis-1998}).

Let $A=K[\gg]$ be the toric $K$-algebra with generators $\gg=g_1,\ldots,g_n$.
It is well know {n} that the dimension of $A$ is the rank of the  matrix $\log(\gg)$. Thus, if all monomials of $\gg$ are of  degree $r$ and the rank of  $\log({\gg})$ is $d$,  then the Gauss algebra $\GG(A)$ of $A$ is a toric algebra generated by monomials of degree $(r-1)d$.
Then  (\ref{minor-monomail}) implies that
\[\GG(A)=K[(g_{i_1}\cdots g_{i_d})/(x_1\cdots x_d) :\
 \ \det(\mathrm{Log}(g_{i_1},\ldots, g_{i_d}))\neq 0].
\]
The injective $K$-algebra homomorphism $ \GG(A)\hookrightarrow A$ is defined by multiplying each generator  of $\GG(A)$ by $x_1\cdots x_d$. Therefore,
\[
\GG(A)\simeq K[g_{i_1}\cdots g_{i_d} :\  \det(\mathrm{Log}(g_{i_1},\ldots, g_{i_d}))\neq 0]\subseteq A.
\]
The morphism $\GG(A)\hookrightarrow A$ induces the rational Gauss map
\[\gamma: \proj{A}\dashrightarrow \proj{\GG(A)} .\]
\begin{Remark}\label{birational}\rm
	Let $A=K[g_1,\ldots, g_n]$ be a standard graded
$K$-subalgebra of $K[x_1,\ldots,x_d]$, up to degree renormalization, and $X=\proj{A}$. Since
\[
X=\bigcup_{i=1}^n \mathrm{Spec}(K[g_1/g_i,\ldots,g_n/g_i]),
\]
	it follows that the field $ K(X)$ of rational functions of $X$  is equal to  the field of fractions  of any of the  algebras
$K[g_1/g_i,\ldots,g_n/g_i]$.

Let  $B\subset A$ be an extension of homogeneous  standard graded algebras, and $X=\proj{A}$ and $Y=\proj{B}$. Let $A$ be a domain. Then
the corresponding dominant rational map $X \dashrightarrow Y$ is birational if and only if $K(X)=K(Y)$.
	
Therefore if $A=K[g_1,\ldots,g_n]$ is the toric algebra as above, then the morphism $\gamma:\proj{A} \dashrightarrow\proj{\GG(A)}$ is birational if and only if for all $i<j$, the fractions $g_i/g_j$ can be expressed as a product of fractions of the form  $(g_{i_1}\cdots g_{i_d})/(g_{j_1}\cdots g_{j_d})$ with $\det(\mathrm{Log}(g_{i_1},\ldots, g_{i_d}))\neq 0$ and $\det(\mathrm{Log}(g_{j_1},\ldots, g_{j_d}))\neq 0$.

For example,  $\gamma: \proj{A} \dashrightarrow\proj{\GG(A)}$ is birational, when $A\subseteq k[s,t]$  is the coordinate ring of the projective monomial curve parametrized by the generators of $A$~\cite[Proposition 3.8]{BGS}.

\end{Remark}
In general,  normality, Cohen-Macaulayness  or other homological or algebraic properties   are  not preserved when  passing from $A$ to $\GG(A)$.  For example, the squarefree $r$-Veronese algebra $A=K[V_{r,d}]$ is normal Cohen-Macaulay, while  for $r=2$ the Gauss algebra $\GG(A)$ is normal and Cohen--Macaulay if and only if $d\geq 5$, see  Theorem~\ref{Veronese}.

The following example shows that the Gauss algebra of a non-normal toric algebra may be   normal.

\begin{Example}\label{non-nomral}
Let $A=K[s^6, s^5t, s^4t^2,s^3t^3, t^6]\subset K[s,t]$ be the homogeneous coordinate ring of the  projective monomial curve embedded in $\PP_K^4$. By \cite[Lemma 3.7]{BGS},  the $K$-algebra  $A$ is not an isolated singularity and hence is not normal. However, the Gauss algebra  $\GG(A)$ is the $8$-Veronese algebra  in $k[t,s]$, which is normal,  Cohen-Macaulay and  an isolated singularity.
\end{Example}

\section{Borel-fixed algebras}

We start with the following lemma which is crucial for the kind of algebras studied in this section.

\begin{Lemma}\label{Tran-det}
Let $g_1,\ldots,g_d\in S = K[x_1,\ldots,x_d]$  be  homogeneous polynomials,  and let $\varphi: S\to S$ be a linear automorphism. Then
\[\det(\Theta(\varphi(g_1),\ldots, \varphi(g_d)) {)}=\det (\varphi)\cdot\varphi(\det(\Theta(g_1,\ldots,g_d)) {)}.\]
\end{Lemma}
\begin{proof}
Consider the linear transformation $\varphi(x_i)=\sum_{j=1}^{d}a_{ji}x_j$, $i=1,\ldots,d$.  For polynomial $g\in  K[x_1,\ldots,x_d] $ a direct computation with
derivatives shows that
\[
\dfrac{\partial \varphi (g)}{\partial x_i}=a_{i1}
\varphi(\dfrac{\partial g}{\partial x_1})+\cdots+ a_{id}
\varphi(\dfrac{\partial g}{\partial x_d}).
\] 	
We have
\begin{eqnarray}
\nonumber \det(\Theta(\varphi(g_1), \ldots, \varphi(g_{ {d}}))) &=& \det\left(
\begin{bmatrix}
a_{11}& \cdots & a_{1d}\\
\vdots & \ddots & \vdots\\
a_{d1}& \cdots & a_{dd}
\end{bmatrix}
\begin{bmatrix}
\varphi(\dfrac{\partial g_1}{\partial x_1})& \cdots & \varphi(\dfrac{\partial g_d}{\partial x_1})\\
\vdots & \ddots & \vdots\\
\varphi(\dfrac{\partial g_1}{\partial x_d})& \cdots & \varphi(\dfrac{\partial g_d}{\partial x_d})
\end{bmatrix}\right)\\
\nonumber \\
\nonumber&=& \det(\varphi).\det(\varphi(\Theta(g_1,\ldots,g_d)))\\
\nonumber &=&  \det(\varphi).\varphi(\det((\Theta(g_1,\ldots,g_d)))).
\end{eqnarray}
\end{proof}

Recall that a set $G=\{g_1,\ldots,g_n\}$ of monomials of the same degree in $K[x_1,\ldots,x_d]$ is called   {\em Borel set}, if the monomial ideal generated by $G$ is fixed under the action of all linear automorphisms
$\varphi: S\to S$ defined by  nonsingular  upper triangular matrices. The ideal generated by a Borel set, is called a Borel-fixed ideal.

  If $\chara(K)=0$, as we always assume in this paper, the Borel-fixed ideals  are just the
strongly stable  monomial ideals, that is, the monomial ideals $I$   with the property that $x_i(u/x_j)\in I$ for the all monomial generators $u$ of  $I$, and all integers $i<j$ such
that $x_j$ divides $u$.
Let $B\subseteq G$. Then the elements of $B$ are called {\em Borel generators} of $G$, if $G$ is the smallest Borel set containing $B$. In this case if $B=\{u_1,\ldots,u_t\}$, we write
$G=\langle u_1,\ldots,u_t\rangle$. For instance, the Borel set  generated by $\{x_1x_3, x_2x_4\}$ is $\langle x_1x_3, x_2x_4\rangle=\{x_1^2,x_1x_2,x_1x_3,x_1x_4,x_2^2,x_2x_3,x_2x_4\}$.
 A Borel set $G$ is called  {\em principal} if  there exists $u\in G$ such that $G=\langle u\rangle$.

Let $G$  be a  Borel set of monomials of degree $r$. The Borel generators of $G$ are characterized by the property that they are maximal among the monomials
of $G$ with respect to the following partial order on  the monomials: let $u= x_{i_1}x_{i_2}\ldots x_{i_r}$ and $v=x_{j_1}x_{j_2}\ldots x_{j_r}$ with $i_1\leq
i_2\leq \ldots \leq  i_r$ and $j_1\leq j_2\leq \ldots \leq  j_r$. Then we set $u\prec v$,  if $i_k\leq j_k$ for $k=1,\ldots,r$. In particular, if $v=x_{i_1}^{c_1}\cdots x_{i_r}^{c_r}$ with $c_i>0$, and $u=x_1^{a_1}\cdots x_d^{a_d}$. Then $u\npreceq v$ if and only if there exists $j$, such that
\begin{eqnarray}\label{borel-order}
a_{i_j+1}+\cdots+a_d\geq c_{j+1}+\cdots+c_r+1.
\end{eqnarray}
	
	Let $G=\{g_1,\ldots,g_n\}\subset K[x_1,\ldots,x_d]$ be a Borel set. Then we call $A=K[g_1,\ldots,g_n]$, a {\em Borel fixed algebra}, if $\dim(A)=d$.
Note that $\dim(A)=d$, if and only if  there exists $j$ such that $x_d|g_j$. Indeed, since $G$ is a Borel set, the  condition implies that $\{x_1^r,x_1^{r-1}x_2,\ldots,x_1^{r-1}x_d\}\subseteq G$, where $r$ is the degree of the monomials in $G$. The  log-matrix of these elements is upper triangular, and so has rank  $d$. This shows that $\dim(A)=d$.
Indeed, $A$ is isomorphic to the polynomial ring $K[x_1,\ldots,x_d]$ by multiplication by $1/x^{r-1}$.

\begin{Theorem}\label{Borel_fix}
 The Gauss algebra of a   Borel-fixed algebra is a   Borel-fixed algebra.
\end{Theorem}
\begin{proof}
Let   $A$ be a  Borel-fixed algebra with monomial  generators  $G=\{g_1,\ldots,g_n\}$. Let  $G'$ be the set of the corresponding  monomial generators of $\GG(A)$. We want to show that $G'$ is a Borel set. For this, it is enough to show that the ideal $I'$ generated by $G'$ is a Borel-fixed ideal.
Let $g$ be a monomial generator in $I'$. Then  $g=\det(\Theta(g_{i_1},\ldots,g_{i_d}))$. Let $I$ be the monomial ideal generated by $G$. By Lemma~\ref{Tran-det}, for any   upper triangular
automorphism $\varphi:\; S\to S$, one has
\[\varphi(g)=\varphi(\det(\Theta(g_{i_1},\ldots,g_{i_d})))=\det(\varphi)^{-1}(\det\Theta(\varphi(g_{i_1}),\ldots,\varphi(g_{i_d}))).\]
Since $I$ is a Borel-fixed ideal, each $\varphi(g_i)$ is a $K$-linear combination of elements of $G$. By using the fact that $\Theta(-)$ is a multilinear function, we get $\varphi(g)\in I'$. This shows that $\GG(A)$ is Borel-fixed.
\end{proof}

\begin{Corollary}[\cite{BGS},  Proposition 3.2]\label{Veronese}
The Gauss algebra of an $r$-Veronese algebra is an $(r-1)d$-~Veronese algebra.
\end{Corollary}
\begin{proof}
Consider  the monomials $g_1= x_1x_d^{r-1},\ldots, g_{d-1}=x_{d-1}x_d^{r-1}, g_d=x_d^r$. As the log-matrix of $\gg$ is  non-singular, the monomial   $g_1\cdots
g_d/x_1\cdots x_d= x_d^{(r-1)d}$ belongs to the Gauss algebra. Since the $r$-Veronese is a Borel-fixed ideal, the assertion follows  from
Theorem~\ref{Borel_fix}.
\end{proof}

In general the number of Borel generators of the Borel fixed algebra $A$ and that of $\GG(A)$ may be different.  In fact, let $\{x_2x_3, x_1x_4\}$ be the set of Borel generators of $A$. Then $A=K[x_1^2,x_1x_2,x_2^2,x_1x_3,x_2x_3,x_1x_4]$ and
the log-matrix of the generators of $A$ is
\[\begin{bmatrix}
	2&1&0&1&0&1\\
	0&1&2&0&1&0\\
	0&0&0&1&1&0\\
	0&0&0&0&0&1
\end{bmatrix}.\]
Therefore $\GG(A)=K[x_1^4,x_1^3x_2,x_1^2x_2^2,x_1x_2^3,x_1^3x_3,x_1^2x_2x_3,x_1x_2^2x_3]$, and $x_1x_2^2x_3$ is the single Borel generator of $\GG(A)$.

\medskip
 However if $A$ is principal Borel, then $\GG(A)$ is principal Borel as well. More precisely we have the following
\begin{Theorem}\label{principal}
	Let  $A$ be a principal Borel-fixed algebra with Borel generator  $m=x_{i_1}^{a_{i_1}}\cdots x_{i_r}^{a_{i_r}}$ with $a_{i_j}>0$, for $j=1,\ldots,r$.  Then $\mathbb{G}(A)$ is a principal Borel-fixed algebra with Borel generator
\[
m'=\frac{m^{i_r}}{x_{i_1}^{i_1-1}x_{i_2}^{i_2-i_1}\cdots  x_{i_{r-1}}^{i_{r-1}-i_{r-2}}x_{i_r}^{i_r-i_{r-1}+1}}.
\]\end{Theorem}

\begin{proof}
	 We first show that $m'\in \GG(A)$.	Let $g_{k,l}=x_l(m/x_{i_k})$, $k=1,\ldots,r$, $l=i_{k-1},\ldots,i_k-1$ for all $k$, where $i_0=1$. Then
	the $g_{k,l}$ belong to $\langle m\rangle$, and
	\[m'=(\prod_{k=1}^{r}\prod_{l=i_{k-1}}^{i_k-1}g_{k,l})/x_1x_2\cdots x_d.
	\]
	
	We order the monomials $g_{k,l}$ lexicographically and consider the corresponding  log-matrix $A$. The $i$th row with $i\notin\{i_1,\ldots,i_r\}$  has only
	one  non-zero entry which is $1$. So in order to compute the determinant of the log-matrix, we reduce to the computation of the cofactor of that nonzero entry, indeed we skip the $i$th row and the column corresponding to the   nonzero
	entry.  This can be done for all $i\notin\{i_1,\ldots,i_r\}$. Then  we obtain the log-matrix $M$ of the following sequence of monomials
	\[m, x_{i_1}\frac{m}{x_{i_2}},\ldots, x_{i_{r-1}}\frac{m}{x_{i_r}}
	\]
	with respect to $x_{i_1},\ldots, x_{i_r}$. Subtracting the first column of $M$ from the other columns of $M$, we obtain the following matrix
	\[
	\begin{bmatrix}
	a_{i_1}& 1&0&\cdots&\cdots & \cdots& 0\\
	a_{i_2}&-1&1&0&\cdots & \cdots& 0\\
	a_{i_3} & 0 & -1& 1 & 0 &\cdots&0\\
	\vdots & \vdots &\ddots& \ddots&\ddots& {\ddots}&\vdots\\
	\vdots & \vdots && \ddots&\ddots& 1 & 0\\
	a_{i_{r-1}}& 0 & \cdots&\cdots& 0 & -1 & 1\\
	a_{i_r}& 0 & \cdots&\cdots&\cdots& 0 & -1\\
	\end{bmatrix}
	\]
	Now for each $i>1$, we add  the  $i$th row  to the first row. The result is a lower triangular matrix  with non-zero entries on the diagonal. This shows that $A$
	is non-singular, and proves that $m'$ is a generator of the Borel-fixed algebra $\GG(A)$.
		
Since  $\GG(A)$ is a Borel-fixed ideal, by  Theorem~\ref{Borel_fix},
	it is enough to prove that for  any monomial $g$ in $\GG(A)$, one has $g\preceq m'$. Let $m=x^{a_{i_1}}_{i_1}\cdots x^{a_{i_r}}_{i_r}$, $m'=x^{a'_{i_1}}_{i_1}\cdots x^{a'_{i_r}}_{i_r}$.  By definition of $m'$, we have
	$a'_{i_j}=i_ra_{i_{j+1}}-i_{j+1}+i_j$ for $j=2,\ldots,r-1$, and $a'_{i_r}=i_r(a_{i_r}-1)+i_{r-1}-1$.
Let    \[g=(\prod^{i_r}_{i=1}g_i)/x_1x_2\cdots x_{i_r},\] where $g_1,\ldots,g_{i_r}$ belong to the minimal monomial generating set
	of $A$,  the latter having a non-singular log-matrix. If $g\npreceq m'$, then by Borel order property (\ref{borel-order}), there exists 	  $1\leq j\leq r-1$, such that $\prod^{i_r}_{j=1}g_j$ is divisible by
	$w=x_{i_j}^{b_{i_j}}\cdots x_{i_r}^{b_{i_r}}$, and
\begin{eqnarray*}
	\sum^{i_r}_{l=i_j+1}b_l-(i_r-i_j)&\geq &1+\sum^{i_r}_{l=i_{j+1}}a'_{i_l}\\
	&=&\sum_{k=j}^{r-2}(i_ra_{i_{k+1}}-i_{k+1}+i_k)+i_r(a_{i_{r}}-1)+i_{r-1}\\
	&=&(\sum^r_{l=j+1}a_{i_l}-1)i_r+i_j.
\end{eqnarray*}
	 Therefore,
		\begin{eqnarray}\label{clv}
		\sum^r_{l=i_j+1}b_{i_l}\geq (\sum^r_{l=j+1}a_{i_l})i_r.
		\end{eqnarray}
	 We may write $g_s$ as a product of monomials $g_s=f_{s}h_{s}$ with  $\supp(f_{s})\subseteq\{i_1,\ldots,i_j\}$ and $\supp(h_{s})\subseteq\{i_{j+1},\ldots,i_r\}$. As  $g_s\preceq x^{a_{i_1}}_{i_1}\cdots x^{a_{i_r}}_{i_r}$, we have $\deg(h_{s})\leq\sum^r_{l=j+1}a_{i_j}$ and, since $w$ divides $g_1\cdots g_{i_r}$, we get
\[
\sum^{i_r}_{l=i_j+1}b_{i_l}\leq \sum^d_{s=1}\deg(h_s)\leq d\sum^r_{l=j+1}a_{i_l}.
\]
Together with (\ref{clv}), it follows  that $\sum^{i_r}_{s=1}\deg(h_s)=i_r\sum^r_{l=j+1}a_{i_l}$ and, this implies $\deg(h_s)=\sum^r_{l=j+1}a_{i_l}$.

Let $L$ be the log-matrix of $g_1,\ldots,g_{i_r}$.  Then the summation of the last $i_r-i_j$ entries of  each column of $L$ is equal to $\sum^r_{l=j+1}a_{i_l}$, and so the summation of the first $j$ entries of each column is equal to $i_r-\sum^r_{l=j+1}a_{i_l}$. This implies that $L$ is singular, a contradiction.
	\end{proof}

\begin{Corollary}\label{Emma}
	Let $A$ be a principal Borel-fixed algebra. Then $\GG(A)$ is normal and for suitable monomial order its defining ideal has a quadratic Gr\"obner basis.
\end{Corollary}
\begin{proof}
	By the above theorem, $\GG(A)$ is a principal Borel fixed algebra. A principal Borel set is a  polymatroid. Therefore  $\GG(A)$ is normal, see \cite[Corollary~6.2]{HH}. In \cite{dN} it is shown  that the principal Borel-fixed sets are sortable, and so $\GG(A)$ has quadratic Gr\"obner basis.
\end{proof}	

\begin{Corollary}
Let $A$ be a Borel-fixed algebra such that $\dim A=\dim\GG(A)=d$. Then the Guass map $\gamma: \proj{A}\dashrightarrow \proj{\GG(A)}$ is birational.
\end{Corollary}
\begin{proof}
By the hypothesis on the dimension of $\GG(A)$, there exists a generator $u$ of $\GG(A)$ such that $x_d|u$. For $1\leq i<j\leq d$ we have
\[\frac{x_i}{x_j}=\frac{x_i(u/x_d)}{x_j(u/x_d)},\]
which implies that $\gamma$ is birational, since any quotient of monomials in $A$ is the product of some of the  $x_i/x_j$, see Remark~\ref{birational}.
\end{proof}

\bigskip

\section{Squarefree Veronese algebras }

 Let $V_{r,d}$ be the set of all squarefree monomials of degree $r$ in $S=K[x_1,\ldots,x_d]$. The $K$-subalgebra
$A=K[V_{r,d}]$ of $S$  is called the \textit{squarefree  $r$-Veronese} algebra. By Proposition~\ref{Veronese}, the Gauss algebra associated to a Veronese
algebra is again a Veronese algebra. The situation  for  squarefree Veronese algebra is more complicated.

\medskip

Denote by $\mathrm{Mon}_S(t,r)$ the set of all monomials $u$  of degree $r$ in $S$, such that $|\supp(u)|\geq t$, where $\supp(u)=\{i  \ : \ x_i|u\}$.

\begin{Proposition}
\label{kirschtorte}
The monomial ideal generated by $\mathrm{Mon}_S(t,r)$ is  polymatroidal. In particular,  the $K$-algebra $K[\mathrm{Mon}_S(t,r)]$ is
normal and Cohen--Macaulay.
\end{Proposition}
\begin{proof}
 The normality of  the $K$-algebra $K[\mathrm{Mon}_S(t,r)]$ follows from \cite[Theorem~6.1]{HH}, once we have shown that the ideal generated by
$\mathrm{Mon}_S(t,r)$ is polymatroidal. Let $u=x_1^{a_1}\cdots x_d^{a_d},v=x_1^{b_1}\cdots x_d^{b_d}\in\mathrm{Mon}_S(t,r)$. By symmetry we may assume that
$a_1>b_1$. Suppose $a_1>1$, then $x_iu/x_1\in \mathrm{Mon}_S(t,r)$ for any $i\neq 1$, and so the exchange property holds.  Next suppose that $a_1=1$, then
$b_1=0$. If $\supp(u)$ has more than $t$ elements, we may replace $x_1$ by any variable $x_i\in\supp(v)$. Finally, suppose that $\supp(u)$ has exactly $t$
elements. Since $x_1\notin\supp(v)$, there exists $x_j\in\supp(v)\setminus\supp(u)$. Replacing $x_1$ by $x_j$, the exchange property is satisfied.
\end{proof}

In the following result we describe  the structure of the Gauss algebra of the squarefree $2$-Veronese algebra $K[V_{2,d}]$. Note that for $d\leq 3$, the Gauss algebra is isomorphic to a polynomial ring.

\begin{Theorem}\label{2-Veronese}
Let $A=K[V_{2,d}]$, with $d\geq4$. Then
 \begin{enumerate}
 	\item [(a)]
 	$\GG(A)=K[\mathrm{Mon}_S(3,4)\setminus\{x_1x_2x_3x_4\}]$, if $d=4$;
 	
 	\item[(b)] $\GG(A)=K[\mathrm{Mon}_S(3,d)]$,  if $d\geq 5$;
 	
 	\item[(c)] the embedding dimension of $\GG(A)$ is $$\mathrm{edim} \GG(k[A])=\left\{
 	\begin{array}{ll}
 	e-1, &  if\, \hbox{d=4,} \\
 	e, & if \, \hbox{d=5,}
 	\end{array}
 	\right.
 	$$
 	where $e=\binom{2d-1}{d}-(d-1)\binom{d}{2}-d$;
 	\item[(d)] the Gauss algebra is a normal Cohen--Macaulay domain, if and only if $d\geq 5$;
 	\item[(e)] the Gauss map $\gamma: \proj{A}\dashrightarrow \proj{\GG(A))}$ is birational.
 \end{enumerate}
\end{Theorem}

For the proof of the theorem, we need the following
\begin{Lemma}[{\cite[Theorem~2.1]{GKS}}]\label{minor}
	Let $G$  {be} a loop-less connected graph  with the same number of vertices and edges. Then the log-matrix of the edge ideal of $G$  is non-singular if and only
	if $G$ contains an odd cycle.
\end{Lemma}

\begin{proof}[Proof of Theorem~\ref{2-Veronese}]
	First we show that any monomial of the form $m=g/x_1\cdots x_d$ belongs to $\mathrm{Mon}_S(3,d)$, where $g=g_{i_1}\cdots g_{i_d}$ is a product of  pairwise distinct elements of $V_{2,d}$. This then yields  the inclusion $\GG(A)\subseteq K[\mathrm{Mon}_S(3,d)]$.
 Suppose that the number of elements in the support of $m$ is less than $3$. Then at least $d-2$ variables have degree $1$ in $g$. Hence $g$ can be written as a product of at most $d-1$ monomials in $A$, which is a contradiction.

Now, to prove (a) and (b),  let $m$ be an element of $\mathrm{Mon}_S(3,d)$. For $d=4$ and $d=5$ the assertions can be shown by direct computations. Let  $d>5$, and first assume that
$m=x_1\cdots x_d$.   If $d$ is odd, then let
\[g_1=x_1x_2, g_2=x_2x_3,\ldots, g_{d-1}=(x_{d-1}x_d), g_d=(x_dx_1).\]
Then the log-matrix of $g_1,\ldots,g_d$ is non-singular by Lemma~\ref{minor}.
If $d$ is even, then let
\[g_1=x_1x_2, g_2=x_2x_3,g_3=x_3x_1,g_4=x_4x_5,g_5=x_5x_6,\ldots, g_{d-1}=(x_{d-1}x_d), g_d=(x_dx_4).\]
Now,  the log-matrix  is
\[
\begin{bmatrix}
A &0\\
0&B\\
\end{bmatrix},
\]
where $A$ and $B$ are incidence matrices  of odd cycles, and so it is non-singular.

Next assume that $m\neq x_1\cdots x_d$. Without loss of generality  we may assume that $m=x_1^{r_1}\cdots x_{d-1}^{r_{d-1}}$. Since $\deg(m)=d$, there exists $i$
such that $r_i>1$. Let $u=m/x_i$. Then $u\in \Mon_{S'}(3,d-1)$, where $S'=K[x_1,\ldots,x_{d-1}]$. By induction, $(x_1\cdots x_{d-1})u=g_1\cdots g_{d-1}$ with
$g_i\in A$ and $L(g_1,\ldots, g_{d-1})$ non-singular. Let $g_d=x_ix_d$. Then $(x_1\dots x_d)m=g_1\cdots g_{d-1}g_d$. Since all the entries of the last row of
$\mathrm{Log}(g_1,\ldots,g_d)$ are zero, except the last one, which is equal to $1$, we see that $\mathrm{Log}(g_1,\ldots,g_d)$ is non-singular.

(c) follows from (b) by a simple counting argument.

(d): If $d\geq 5$,  it follows from (b) and Proposition~\ref{kirschtorte} that $\GG(A)$ is normal, and Cohen--Macaulay by  Hochster \cite{Ho}. On the other hand a calculation with Singular~\cite{DGPS} shows  that for $d=4$,  the $h$-vector of  $\GG(A)$ has a negative component. Therefore, in this case $\GG(A)$ is not Cohen--Macaulay.

(e): By Remark~\ref{birational}, it is suffices to show that for  every $1\leq i<j\leq n$,
\[K[\dfrac{x_ix_j}{x_rx_s} \ | \  1\leq r<s\leq n] \subset  K[\frac{u}{v}\ | \ u,v\in
\mathrm{Mon}(3,d)  ]. \]

For $1\leq i<j\leq d$ one has
\begin{eqnarray}
\label{quotient}
\dfrac{x_i}{x_j}=\dfrac{x_k x_l^{d-2}x_i}{x_k x_l^{d-2}x_j},
\end{eqnarray}
with $i,j,k,l$ pairwise distinct. Hence $(x_ix_j)/(x_rx_s)$ has an expression as in  (\ref{quotient}), if $\{i,j\}\sect \{r,s\}\neq \emptyset$. Otherwise,
\[\frac{x_ix_j}{x_rx_s}=\frac{(x_jx_s^{d-2}x_i)(x_ix_r^{d-2}x_j)}{(x_jx_s^{d-2}x_r)(x_ix_r^{d-2}x_s)}.
\]

\begin{Remark}
	(a) Let $A=K[V_{r,d}]$. We may assume that $d\geq r+2$, otherwise $\GG(A)$ is a polynomial ring.  Then
	\[
\GG(A)\subseteq\{x_1^{a_1}\cdots x_d^{a_d}\in\Mon_{S}(r+1,(r-1)d)\ : \ a_i\leq d-2 \quad \text{for}\quad 1\leq i\leq d\}.
\]
For $r=2$, the equality holds if and only if $d\geq 5$. It would be interesting to know for which  $r>2$ and $d$ the equality holds.
	
	(b) According to White's conjecture \cite{W}, the base ring of a polymatroid  is generated by the so-called exchange relations, which are quadratic binomials. Since $\Mon_{S}(3,d)$ is polymatroidal, we expect that the Gauss algebra of $K[V_{2,d}]$ has quadratic relations.
\end{Remark}

\end{proof}

		\section{Edge rings}
		
		Let $G$ be a simple graph on the vertex set $V(G)=[d]$ and edge set $E(G)=\{e_1,\ldots,e_m\}$.
		 For given subset $V\subseteq [d]$, we set $x_V=\prod_{i\in V}x_i$. In the case that $V$ is an edge $e=\{i,j\}$, we simply write $e$ instead of
$x_V=x_ix_j$.
		 The edge ideal $I(G)$, of $G$, is the  ideal generated by the monomials $e\in E(G)$.
		 Note that the log-matrix of $E(G)$ is   the incidence matrix of $G$.
		
		 Let $V\subseteq V(G)$ and $E\subseteq E(G)$ with $| V|=|  E|$. We denote by $\Delta_{V,E}$  the minor of the log-matrix $\log(E(G))$, with rows $V$ and
columns $E$.

\begin{Lemma}\label{Lemma} Let $V\subseteq V(G)$ and $E\subseteq E(G)$ with $|V|=|E|=r$, and let $\Delta_{V,E}$ be the minor of the log-matrix $\log(E(G))$, with
rows $V$ and columns $E$. 	 Suppose the edges in $E$ can be labeled as $e_1,\ldots,e_r$, such that
	\begin{eqnarray}\label{maybe}
		| V\cap(e_1\cup\cdots\cup e_i)| =i \quad \text{for} \quad 	i=1,\ldots,r.
		\end{eqnarray}
	Then  		$\Delta_{V,E}\neq 0$. The converse holds  if $G$ is a bipartite graph.
\end{Lemma}
		\begin{proof}
Suppose condition~(\ref{maybe}) holds. Let $M$ be the matrix with rows $V$ and columns $E$. Let $V\cap e_1=\{v\}$. Then the first column of $M$ has only one
non-zero entry, corresponding to vertex $v$. Let $V'=V\setminus\{v\}$, then
\[|V'\cap\{e_2,\ldots,e_i\}|=i-1 \quad \text{for}\quad i=2,\ldots,r.\]
Now, by the induction hypothesis the matrix $M'$ whose rows are $V'$ and whose columns are $e_2,\ldots,e_r$, is non-singular.  It follows that $M$ is
non-singular.

Conversely, assume that $\Delta_{V,E}\neq 0$. Then we claim that there exists a column $e_1$ in $E$ such that $| V\cap e_1|=1$. Indeed, if $| V\cap e_i|>1$, for
$i=1,\ldots,r$, then $M$ is the incidence matrix of a bipartite graph. Now, Lemma~\ref{minor} implies $\Delta_{V,E}=0$, contradiction.  Let
$V'=V\setminus\{v_1\}$, where $V\cap e_1=\{v_1\}$. Then the matrix $M'$ whose rows are $V'$ and whose columns are $e_2,\ldots,e_r$, is non-singular.
Now, $|V'\cap\{e_2,\ldots,e_i\}|=i-1$ for $i=2,\ldots,r$, by induction. This implies that
$| V'\cap\{e_1,\ldots,e_i\}|=i$ for $i=2,\ldots,r$.
\end{proof}

\begin{Corollary}\label{schwarzwaelder}
	Let $G$ be a graph with $c$ connected components, and $V\subset V(G)$, with $| V|\leq d-c$. Then there exists $E\subseteq E(G)$ with $|
	E|=|V|$ such that $\Delta_{V,E}\neq 0$.
\end{Corollary}
\begin{proof}
	Let $| V|=r$. Since $r\leq d-c$, we can choose a set $E$ of $r$ edges such that $e\cap V\neq\emptyset$ for each $e\in E$.  	Now, the matrix $M$ with rows $V$ and
	columns $E$, is not the incidence  matrix of a forest, since for a forest the number of vertices is strictly bigger than  the number of edges.
	Hence there exists an edge $e_1$ in $E$ such that $e_1\cap V=\{v\}$.
	Removing the edge $e_1$ from $G$, the number of connected components $c'$ of  $G\setminus\{e_1\}$ is at most $c+1$. Let $V'=V\setminus\{v\}$.
	Then $|V'|\leq d-c'$. By induction  there exist edges
	$e_2,\ldots,e_r$ such that $|V'\cap e_2\cup\cdots\cup e_i|=i-1$ for $i=2,\ldots,r$. It follows that $e_1,\ldots,e_r$ satisfies the condition (\ref{maybe}).
	Therefore the desired result follows from Lemma~\ref{Lemma}.
\end{proof}

	Let $G$ be a simple bipartite graph. Let $L$  {be} a non-empty subset  of $[d]$,   and let $G^L$ be the graph which is obtained from $G$  by attaching a loop to
$G$ at each vertex  belonging to $L$.  	
	For given set $T\subseteq E(G)$, let $G(T)$ denote the graph with $V(G(T))=V(G)$ and $E(G(T))=T$.

	\begin{Theorem}\label{notneeded} 	Let $G$ be a  bipartite graph with $r$ components and $L$ be a subset of $[d]$.  Let $A$ be
		the edge ring of $G^L$. Then the following statements hold.
		\begin{enumerate}
			\item[\rm (a)]  $A$ has dimension $d$ if  and only if $L$  contains  at least one vertex of each component of $G$.
			\item[(b)] If the condition (a)   {is} satisfied, then
			the Gauss algebra $\GG(A)$ is generated  by the monomials
			\[
			g_{V,T}= x_{V}\frac{e_{T}}{x_{V^c}},
			\]
			where $V$ is a non-empty subset of $L$,   $V^c=[d]\setminus V$, and $e_T=\prod_{e\in T}e$ where $T\subseteq E(G)$ satisfies
						\begin{enumerate}
			 \item[(i)] $G(T)$ is a  forest {, which may have isolated vertices as some of its connected components};	
			 		 \item[(ii)]  each connected component of $G(T)$ contains exactly one vertex of $V$.
			\end{enumerate}
		\medskip
In particular, when $|V|=1$, then  the cardinality of the minimal set of generators of $\GG(A)$  is bounded by the number of spanning trees of $G$. Moreover, $g_{V,T}=g_{V,T'}$ if and only if each vertex of $G$ has the same degree in $T$ and $T'$.
		\end{enumerate}
	\end{Theorem}
	\begin{proof}
				(a).  {As $G$ is a bipartite graph, the log-matrix of $G$ is singular by Lemma~\ref{minor}.  }
				We show that the log-matrix of $G^L$ has a non-singular maximal minor, if and only if $L$  contains  at least one vertex of each component of $G$. Let
		$L_i=G_i\cap L$, then $G_1^{L_1},\ldots,G_r^{L_r}$ are the connected components of $G^L$, and the log-matrix of $G^L$ has maximal rank if and only if the
		log-matrix of each $G_i^{L_i}$ has maximal rank.  {Therefore, it is enough to show that the log-matrix of a connected graph $G$, with at least one loop, is non-singular. Assume that there is a loop at vertex $1$. Then the $1$st column has only one non-zero entry at $1$st row.
		Let $A$ denote the log-matrix of  $G$ and $|V(G)|=n$.  In order to compute the rank  of $A$, we may skip the $1$st row and the $1$st column, obtaining a new matrix $A_1$, which has maximal rank $n-1$, by  Corollary~\ref{schwarzwaelder}. Therefore the rank of $A$ is equal to $n$. }
		
			\medskip
		
		(b).
		We first show that the conditions (i),(ii) are equivalent to
			\begin{enumerate}
			\item[($\alpha$)]  $|T|=|V^c|$;
			\item[($\beta$)] the elements of $T$ can be labeled as $e_1,\ldots,e_m$   such that
			\[|V^c\cap(e_1\cup\cdots\cup e_i)|=i \quad  \text{for} \quad 	i=1,\ldots,m.\]
		\end{enumerate}
Suppose that (i) and (ii) are satisfied. If $T=\emptyset$, then the equivalence of (i),(ii) with $\alpha,\beta$ is trivial. Now, assume that $T\neq\emptyset$, and let  $G(T)_1,\ldots,G(T)_t$ be  the connected components of $G(T)$ with $|V(G(T)_i)|\geq 2$   and $v_i$ be the vertex of $V$ belonging to $G(T)_i$.  Since $G(T)_i$ is a tree, we may label the edges of $G(T)_i$ as $e_{i_1},\ldots,e_{i_{s_i}}$ such that $v_i\in e_{i_1}$ and $|e_{i_j}\cap(e_{i_1}\cup\cdots\cup e_{i_{j-1}})|=1$ for all $j=1,\ldots,s_i$. Then the sequence of edges
		\[e_{1_1},\ldots,e_{1_{s_1}},e_{2_1},\ldots,e_{2_{s_2}},e_{3_1},\ldots\]
		satisfies conditions ($\alpha$),($\beta$).
		
		Conversely, condition $(\beta)$ guarantees  that $G(T)$ does not contain any  cycle, and so it  is a forest, which by $(\alpha)$ has $d-|V|$ edges. Therefore $|V|$ is equal to the number of connected components of $G(T)$. let  $G(T)_1,\ldots,G(T)_t$ be  the connected components of $G(T)$ with $|V(G(T)_i)|\geq 2$, and let  $e_{i_j}$ be the first edge, with respect to the labeling in ($\beta$), such that $e_{i_j}\cap G(T)_j\neq\emptyset$. Then $|e_{i_j}\cap V|=1$, so each connected component of $G(T)$ contains at least one vertex in $V$. Since $G(T)$ has $|V|$ number of components, each component should contain exactly one element of $V$.

		\medskip

Let $g$ belong to the minimal set of generators of $\GG(A)$. Then $g=\frac{g_1\cdots g_d}{x_1\cdots x_d}$, where $g_i$ is a monomial generator of $A$ and the log-matrix of $g_1,\ldots,g_d$ is non-singular.  Since the incidence matrix of a bipartite graph is singular by Lemma~\ref{minor}, at least one $g_i$ corresponds to a loop. After relabeling, we may assume that $\{g_1,\ldots,g_d\}=\{x_1^2,\ldots,x_s^2,e_1,\ldots,e_{d-s}\}$ and
	\[\ \log(x_1^2\cdots x_s^2e_{1}\cdots e_{{d-s}})=
\begin{bmatrix}
2&0&\cdots&0&a_{1,1}&\cdots&a_{1,d-s}\\
0&2&&0&a_{2,1}&\cdots&a_{2,d-s}\\
\vdots&&\ddots&\vdots&\vdots&\vdots\\
0&0&\cdots&2&a_{s,1}&\cdots&a_{s,d-s}\\
0&0&\cdots&0&a_{s+1,1}&\cdots&a_{s+1,d-s}\\
\vdots&\vdots&\vdots&\vdots&\vdots&\vdots&\vdots\\
0&0&0&0&a_{d,1}&\cdots&a_{d,d-s}
\end{bmatrix},
\]
where $A=[a_{r,t}]$ is the log-matrix of $e_{1}\cdots e_{{d-s}}$.
Let $V=\{1,\ldots,s\}$ and $T=\{e_1,\ldots,e_{d-s}\}$.  	Then \[g=g_{V,T}=\frac{x_1^2\cdots x_s^2e_{1}\cdots e_{{d-s}}}{x_1\cdots
	x_d}=x_1\cdots x_s\frac{e_{1}\cdots e_{{d-s}}}{x_{s+1}\cdots x_{d}},\] and
the log-matrix is non-singular if and only if the submatrix $A'$ with  rows
 $s+1,\ldots,d$ and  columns $s+1,\ldots, d$, is non-singular, and by Lemma~\ref{Lemma}, $A'$ is non-singular if  and only if the condition  ($\beta$)
 is satisfied.		
\end{proof}

\begin{Example}
Let $G$ be a path graph  with $d$ vertices,  and edges $\{1,2\},\{2,3\},\ldots,\{d-1,d\}$.
Let $L\subseteq [d]$ and $A$ be the edge ring of $G^L$.  {The induced subgraph by any set $T\subset E(G)$,  can be considered as a disjoint union of intervals. Since $G$ is a path graph, $G(T)$ is a forest. By Theorem \ref{notneeded}, the product of edges in $T$ should be divisible by all vertices in $[d]\setminus V$. Therefore $G(T)$, covers $[d]$. In other words,
the generators of $\GG(A)$ correspond to  interval partitions of  $[d]$ with the property that each interval contains exactly one element of $V$. Let  $[d]=\cup^r_{i=1}[a_i,b_i]$ with $[a_i,b_i]\sect[a_j,b_j]=\emptyset$ for all $i\neq j$,  and with   $|V\cap[a_i,b_i]|=1$ for $i=1,\ldots,r$. The  corresponding generator of $\GG(A)$, is
 \begin{eqnarray}\label{easy}
	  (\prod^r_{i=1}\prod_{j\in ]a_i,b_i[ }x_j)\prod_{j\in L}x_j^{c_j},
\end{eqnarray}
where $c_j=2$, if $j$ belongs to a proper interval, and is $c_j=1$,  otherwise.
Here   $[a_i,b_i]$ is said to be proper if $b_i-a_i>0$. }

The above discussions show that if $L=[d]$, then the number of generators of $\GG(A)$ is the number  {
\[
\lambda_d=\sum^{d}_{r=1}\sum_{a\in P_r}\prod_{i=1}^r(a_{i+1}-a_i),
\]	
 where $P_r=\{a=(a_0,\ldots,a_r) : 0=a_0<a_1< a_2<\cdots <a_r=d)\}$, for $r=1,\ldots,d$. }
   The sequence $(\lambda_d)_{d\geq 1}$ begins as follows \[1,3,8,21,55,144,377,\ldots.\] The recursive formula $\lambda_d=3\lambda_{d-1}-\lambda_{d-2}$, describes the beginning of the sequence. This  seems to be the rule for the whole sequence $(\lambda_d)_{d\geq 1}$.

\medskip	
In the case that  $L=\{i<j\}$,
$\GG(A)$ is generated by $j-i+2$ monomials
\[
x^2_ix_2\cdots x_{d-1} \ ,\ x_j^2x_2\cdots x_{d-1} \ , \ 	(x_i^2x_j^2x_2\cdots x_{d-1})/x_{k}x_{k+1} \quad \text{for} \quad i\leq k\leq j-1.
\]
 An easy calculation shows that the log-matrix of $\GG(A)$ has rank $j-i+1$. In particular $\GG(A)$ is a hypersurface ring. It can be shown that the multiplicity of $\GG(A)$ is $j-i$. When $i=1,j=d$, the defining equation of $\GG(A)$ is
\[f=\left\{\begin{array}{ll}
\prod^{d/2}_{i=1}y^2_{2i}-y_1y_{d+1}\prod_{i=1}^{d/2-1}y^2_{2i+1}, & \text{ if } d \text{ is even; }\\ \\
y_1\prod^{(d-1)/2}_{i=1}y^2_{2i+1}-y_{d+1}\prod_{i=1}^{(d-1)/2}y^2_{2i+1}, & \text{ if } d \text{ is odd, }\end{array}
\right.\]	
and if $j=i+1$, then the defining equation is quadratic.
By computing the singular locus, we see that $\GG(A)$ is normal if and only if $d=2$.
\end{Example}
	
\begin{Example}
	Let $G$ be a cycle with $d$ vertices,  and edges $\{1,2\},\{2,3\},\ldots,\{d-1,d\},\{1,d\}$. Let $L=\{1\}$ and $A$ be the edge ring of $G^L$.
	When $d$ is even, the	spanning trees of $G$  correspond to the generators of $\GG(A)$.
	Each spanning tree of $G$ is obtained by removing one edge from $G$, and so the  generators of $\GG(A)$ are
	\begin{eqnarray}\label{easy-cycle}
x_1^2\prod^d_{i=3}x_i \ , \ x_1^2\prod^{d-1}_{i=2}x_i \ , \
x_1^3\frac{\prod^d_{i=2}x_i}{x_jx_{j+1}} \quad \text{for} \quad j=2,\ldots,d-1.
	\end{eqnarray}
	 When $d$ is odd, in addition to the monomials in (\ref{easy-cycle}), $\GG(A)$ has one more generator, namely $x_1\cdots x_d$.
		For even $d$, $\dim(\GG(A))=d-1$, and for odd $d$, $\dim(\GG(A))=d$. Hence in both cases  $\GG(A)$ is a hypersurface ring with defining equation
	 \[f=\left\{\begin{array}{ll}
	 y_{d-1}\prod^{d/2-1}_{i=1}y_{2i}-y_{d}\prod_{i=1}^{d/2-1}y_{2i-1}, & \text{ if } d \text{ is even; }\\ \\
	 y_d\prod^{(d-1)/2}_{i=1}y_{2i}-y_{d+1}\prod_{i=1}^{(d-1)/2}y_{2i-1}, & \text{ if } d \text{ is odd. }\end{array}
	 \right.\]	
The initial monomial of $f$ (with respect to  any monomial order) is squarefree. Therefore,   $\GG(A)$ is normal.
\end{Example}

	\begin{Remark}
 Let $G$ be a bipartite  graph on $[d]$, $L=\{i\}$  and $A$ be the edge ring of $G^L$.
The above examples and computational evidence indicate that 	$\GG(A)$ is a hypersurface ring of dimension $d-1$ if and only if $G$ is an even cycle.
\end{Remark}
\begin{Example}\label{bipartite}
Let $G=K_{n,m}$ be a complete bipartite graph with partition sets $X=\{x_1,\ldots,x_n\}$ and $Y=\{y_1,\ldots,y_m\}$. Let  $A$ be the edge ring of $G$ with one loop at vertex $x_1$.
Then
\[\GG(A)=K[x_1^2(x_1,\ldots,x_n)^{m-1}(y_1,\ldots,y_m)^{n-1}].\]
Indeed,  any generator of $\GG(A)$ can be written as $x_1^2(e_1\cdots e_{n+m})/x_1\cdots x_ny_1\cdots y_m$, where $e_j$ is an edge of $G$, which follows  $\GG(A)\subseteq  K[x_1^2(x_1,\ldots,x_n)^{m-1}(y_1,\ldots,y_m)^{n-1}]$.

Let $f$ be a monomial in the generating set of  $(x_1,\ldots,x_n)^{m-1}(y_1,\ldots,y_m)^{n-1}$. Then  $f=x_{i_1}\cdots x_{i_{m-1}}y_{j_1}\cdots y_{j_{n-1}}$, for some $1\leq i_1\leq \cdots\leq i_{m-1}\leq n$ and $1\leq j_1\leq \cdots \leq j_{n-1}\leq m$. Now, let $T$ be the subgraph of $G$ with $V(T)=V(G)$ and $E(T)$ equal to
\[
\{e_1=x_1y_{j_1},\ldots, e_{n-1}=x_{n-1}y_{j_{n-1}},e_{n}=y_1x_{i_1},\ldots,e_{n+m-2}=y_{m-1}x_{i_{m-1}}, e_{n+m-1}=x_ny_n\}.
\]
Then $T$  is a spanning tree of $G$, which implies that
 \[
x_1^2f=x_1^2e_1\cdots e_{n+m-1}/x_1\cdots x_ny_1\cdots y_m
\]
 is a generator of $\GG(A)$.

As a consequence,  the embedding dimension of $\GG(A)$  is $\binom{m+n-2}{n-1}\binom{m+n-2}{m-1}$. However, the number of spanning tress of $G$ is $n^{m-1}m^{n-1}$, see \cite[Theorem~1]{HW}. Therefore, among the spanning trees of $G$, many of them correspond to the same generator in $\GG(A)$.
\end{Example}

\end{document}